\documentclass[preprint,12pt]{elsarticle}

\usepackage{graphicx}
\usepackage{subfigure}
\usepackage{amssymb, amsmath, amsthm}
\usepackage{marvosym}
\usepackage{url}
\usepackage{float}

\newtheorem{theorem}{Theorem}
\newtheorem{lemma}[theorem]{Lemma}
\newtheorem{corollary}[theorem]{Corollary}

\makeatletter
\def\fnum@figure#1{\figurename\nobreakspace\thefigure.}
\makeatother

\journal{Journal of XXX}
\begin{document}
\begin{frontmatter}


\title{ Quasi-Newton method of Optimization is proved to be a steepest descent method under the ellipsoid norm}


\author[TELECOM]{Jiongcheng Li\corref{cor}}
\cortext[cor]{Corresponding author}
\ead{115120196@qq.com}
\address[TELECOM]{China Telecom, Guangzhou, Guangdong, PR China}

\begin{abstract}
    Optimization problems, arise in many practical applications, from the view points of both theory and numerical methods. Especially, significant improvement in deep learning training came from the Quasi-Newton methods. Quasi-Newton search directions provide an attractive alternative to Newton's method in that they do not require computation of the Hessian and yet still attain a superlinear rate of convergence. In Quasi-Newton method,  we require Hessian approximation to satisfy  the secant equation.  In this paper,  the Classical Cauchy-Schwartz Inequality is introduced, then more generalization are proposed. And it is seriously proved that  Quasi-Newton method is a steepest descent method under the ellipsoid norm.
\end{abstract}

\end{frontmatter}

\renewcommand{\theequation}{\arabic{section}.\arabic{equation}}
\renewcommand{\thetable}{\arabic{section}.\arabic{table}}
\renewcommand{\figurename}{Fig.\hspace{-0.2em}}

\section{Introduction}
\label{S1}
\setcounter{equation}{0}
\setcounter{table}{0}

    Optimization is an important tool in decision science. Optimization problems, arise in many practical applications. Especially, signiﬁcant improvement in deep learning training came from the Quasi-Newton methods. The steepest descent method is one of the simplest and the most fundamental minimization methods for unconstrained optimization. The method is a line search method that moves along  $ - \nabla f _k $ at every step. 
One advantage of the steepest descent direction is that it requires calculation of the gradient but not of second derivatives. However, it can be excruciatingly slow on difﬁcult problems. Quasi-Newton methods, like steepest descent, require only the gradient of the objective function to be supplied at each iteration.
Quasi-Newton search directions provide an attractive alternative to Newton's method in that they do not require computation of the Hessian and yet still attain a superlinear rate of convergence. In Quasi-Newton method,  we require Hessian approximation to satisfy  the secant equation.
In this paper,  the Classical Cauchy-Schwartz Inequality is introduced, then more generalization are proposed. And it is seriously proved that  Quasi-Newton method is a steepest descent method under the ellipsoid norm.

\textbf{Classical Cauchy-Schwartz Inequality}:  if x and y are vectors in $ R ^n $, then
$ \left| x ^T y \right| \le \left| \left| x \right| \right| _2 \left| \left| y \right| \right| _2 $
, the equality holds if and only if x and y are linearly dependent.

\section{Make generalizations about Cauchy-Schwartz Inequality}
\label{S2}
\setcounter{equation}{0}
\setcounter{table}{0}

\begin{lemma}
    $ \forall x > 0 $, $ ln \left( x \right) \le x - 1 $ , the equality holds if and only if $ x = 1 $.
\end{lemma}

\begin{proof}
    Let $ f \left( x \right) = \left( x - 1 \right) - ln \left( x \right) $, we have $ f \left( 1 \right) = 0, f' \left( x \right) = 1 - \frac {1} {x} $.

    (1) During $ x > 1 $: $ f' \left( x \right) > 0 $, hence $ x - 1 > ln \left( x \right) $.

    (2) During $ x < 1 $: $ f' \left( x \right) < 0 $, hence $ x - 1 > ln \left( x \right) $. 

    \begin{figure}[H]
        \begin{center}
        \includegraphics[width=10.0cm]{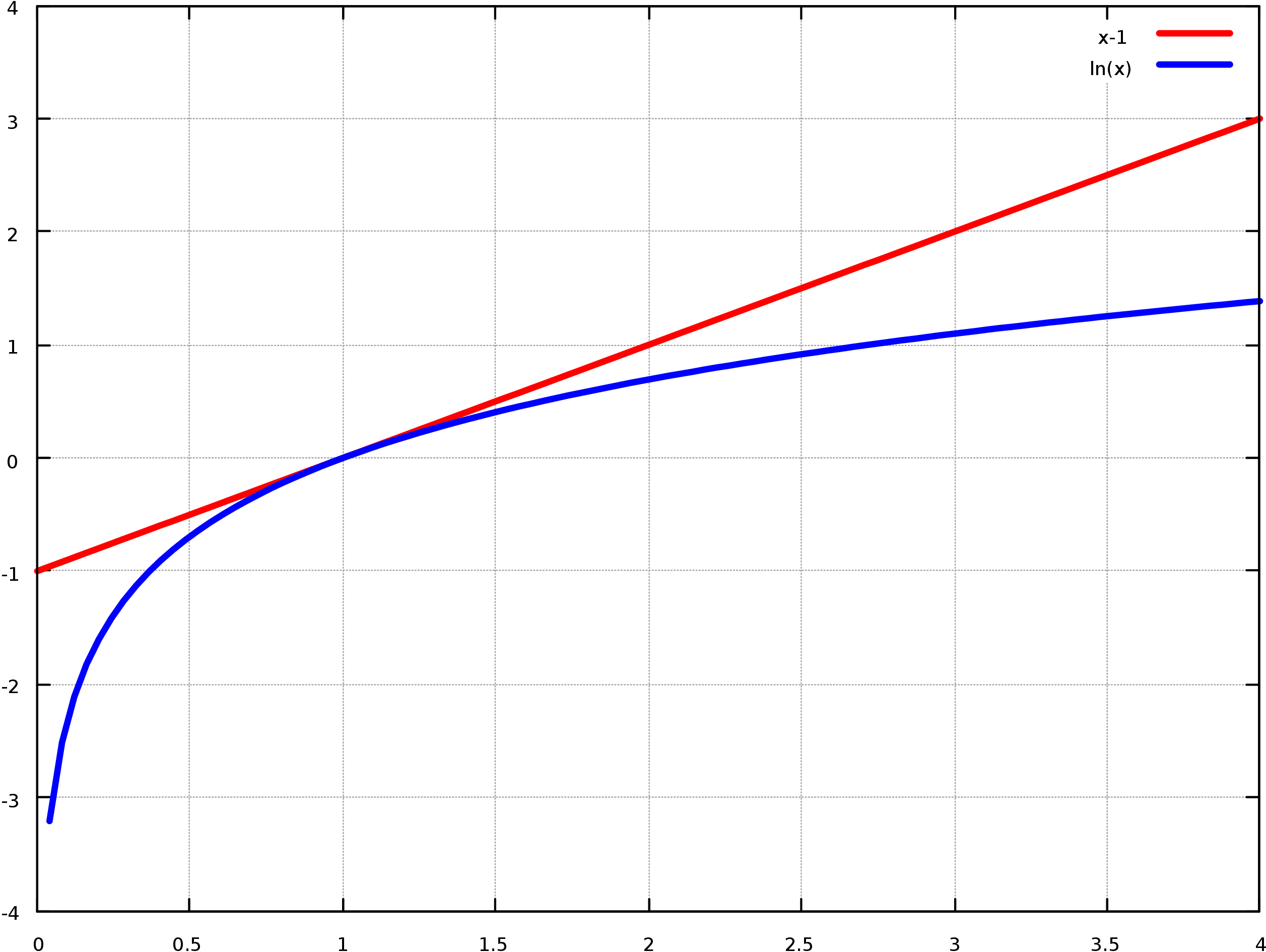}\\
            \caption{x-1 and ln(x)}\label{Fig1}
        \end{center}
    \end{figure}

\end{proof}

You can reconstruct the figure by the gnuplot command \cite{Philipp2016}:

\textbf{gnuplot -p -e \texttt{"} set xrange [0:4]; set yrange [-4:4]; set grid; plot x-1 lc rgb `red` lw 3, log(x) lc rgb `blue` lw 3 title `ln(x)` \texttt{"}}

when you're using gnuplot's built-in functions, log(x) refers to the natural logarithm; you must use log10(x) for the logarithm to base 10.

\begin{lemma}
    (The Weighted Arithmetic Mean–Geometric Mean Inequality) Let $ a_1, \cdots, a_n $ be positive real numbers, $ \delta _1, \cdots, \delta _n $ be positive weights with $ \delta _1 + \cdots + \delta _n = 1 $, then 
    $ a_1 ^{\delta _1} \cdots a_n ^{\delta _n} \le \delta _1 a_1 + \cdots + \delta _n a _n $, 
    and equality holds if and only if all $ a _i $ are equal.
\end{lemma}

\begin{proof}
    Let
    \begin{equation}
        R = \sum _{i = 1} ^n \delta _i a _i
    \end{equation}

    It follows immediately from the lemma 1 that
    \begin{equation}
        \forall i, \frac {a _i} {R} - 1 - ln \left( \frac {a _i } {R} \right)  \ge 0
    \end{equation}

    Multiplying every item by $ \delta _i $, and taking the sum over $ i $ yields:
    \begin{equation}
        \sum _{i=1} ^n \left( \frac {\delta _i a _i} {R} - \delta _i \right) - \sum _{i = 1} ^n \delta _i ln \left( \frac {a _i} {R} \right) \ge 0
    \end{equation}

    It follows from (2.1) that $ \sum _{i = 1} ^n \frac {\delta _i a _i} {R} = 1 $, and with known conditions we obtain $ \sum _{i = 1} ^n \delta _i = 1 $.
    Using (2.3), we have :
    \begin{equation}
        \sum _{i = 1} ^n \delta _i ln \left( \frac {a _i} {R} \right) \le 0
    \end{equation}
    \begin{equation*}
        exp \left[ \sum _{i = 1} ^n \delta _i ln \left( \frac {a _i} {R} \right) \right] \le exp \left( 0 \right) = 1
    \end{equation*}
    \begin{equation*}
        exp \left[ \sum _{i = 1} ^n ln \left( \frac {a _i ^{\delta _i} } {R ^{\delta _i} } \right) \right] \le 1
    \end{equation*}
    \begin{equation*}
        exp \left[ ln \left( \frac {a _1 ^{\delta _1} \cdots a _n ^{\delta _n}} {R ^{\delta _1} \cdots R ^{\delta _n}} \right) \right]  \le 1
    \end{equation*}

    Therefore, 
    \begin{equation*}
        \frac {a _1 ^{\delta _1} \cdots a _n ^{\delta _n}} {R} \le 1 
    \end{equation*}
    . So,
    \begin{equation}
        a _1 ^{\delta _1} \cdots a _n ^{\delta _n} \le \delta  _1 a _1 + \cdots + \delta _n a _n
    \end{equation}

    The equality of (2.5) holds if and only if  equality of (2.4) holds. When $ a _1 = \cdots = a _n $, $ R = a  _1 $, $ \forall i, \frac {a _i} {R} = 1 $, where  equality of (2.4) holds.
\end{proof}

\begin{corollary}
    The Arithmetic Mean–Geometric Mean Inequality follows from Lemma 2 by setting 
    $ \delta _1 = \cdots = \delta _n = \frac {1} {n} $. Furthermore, taking $ n=2 $, we have "Geometric mean not greater than arithmetic mean", which can be found in elementary mathematics.
\end{corollary}

\begin{lemma} 
        (Young inequality) Assume that real numbers p and q are each larger than 1, and $ \frac 1 p + \frac 1 q = 1 $ .  If x and y are also real numbers, then:

    $ x y \le \frac {x ^p} p + \frac {y ^q} q $, and equality holds if and only if $ x ^p = y ^q $.
\end{lemma} 

\begin{proof}

    Setting $ \delta _1 = \frac 1 p, \delta _2 = \frac 1 q $ in the Weighted Arithmetic Mean–Geometric Mean Inequality, we obtain $ \delta _1 + \delta _2 = 1 $.

    Further setting $ a _1 = x ^p, a _2 = y ^q $, we have

    \begin{equation*}
        x y = a _1 ^{\frac 1 p} a _2 ^{\frac 1 q} \le \frac {a _1} p + \frac {a _2} q = \frac {x ^p} p + \frac {y ^q} q
    \end{equation*}

    The equality holds if and only if $ a _1 = a _2 $, that is to say, $ x ^p = y ^q . $

\end{proof}

\begin{corollary}
    It follows immediately from Young inequality  by setting $ x = \sqrt {a}, y = \sqrt {b}, p = q = 2 $ that "Geometric mean not greater than arithmetic mean", which can be found in elementary mathematics.
\end{corollary}

\textbf{We will make two generalizations about Cauchy-Schwartz Inequality. One generalization takes account of power, the other one takes vector norm.}

\begin{theorem}
    (Hölder inequality) if x and y are vectors in $ R ^ n $, then
     
    \begin{equation*}
        \left| x ^T y \right| \le \left| \left| x \right| \right| _p \left| \left| y \right| \right| _q = \left( \sum _{i = 1} ^n \left| x _i \right| ^p \right) ^ {\frac 1 p}   \left( \sum _{i = 1} ^n \left| y _i \right| ^q \right) ^ {\frac 1 q} 
    \end{equation*}

    where p and q are real numbers larger than 1 and satisfy $ \frac 1 p + \frac 1 q = 1 $.
\end{theorem}

\begin{proof} 
    If $ x = 0 $ or $ y = 0 $, the result is obviously true. We assume that both x and y are not zero.

    When $ i \in 1, \cdots, n $,  from Young inequality, we have:
    \begin{equation}
        \frac {\left| x _i y _i  \right|} {\left| \left| x \right| \right| _p \left| \left| y \right| \right| _q}  = \frac {\left| x _i \right|} {\left| \left| x  \right| \right| _p}  \frac {\left| y _i \right|} {\left| \left| y  \right| \right| _q}   \le  \frac 1 p \frac {\left| x _i \right| ^p} {\left| \left| x \right| \right| _p ^p} + \frac 1 q \frac {\left| y _i \right| ^q} {\left| \left| y \right| \right| _q ^q}
    \end{equation}

    Taking the sum over i on both sides of the above inequality yields
    \begin{equation}
        \frac 1 {\left| \left| x \right| \right| _p \left| \left| y \right| \right| _q} \sum _{i = 1} ^n \left| x _i y _i \right|  \le \frac 1 {p \left| \left| x  \right| \right| _p ^p} \sum _{i = 1} ^n \left| x _i \right| ^p + \frac 1 {q \left| \left| y  \right| \right| _q ^q} \sum _{i = 1} ^n \left| y _i \right| ^q   = \frac 1 p + \frac 1 q = 1
    \end{equation}

    Multiplying $ \left| \left| x \right| \right| _p \left| \left| y \right| \right| _q $ on both sides of (2.7), we obtain the result. 
\end{proof}

\begin{corollary}
    (Cauchy-Schwartz Inequality) The Cauchy-Schwartz Inequality follows from Theorem 6 by setting $ p = q = 2 $.
\end{corollary}

\begin{corollary} 
    $ \left| \left| A ^T \right| \right| _2 = \left| \left| A \right| \right| _2 $
\end{corollary} 

\begin{proof} 
    \begin{equation*}
        \left| \left| A ^T \right| \right| _2 = \sigma _{max} \left( A ^T \right),
        \left| \left| A \right| \right| _2 = \sigma _{max} \left( A \right) .
    \end{equation*}

    If A has singular value decomposition $ U \Sigma V ^T $, then $ A ^T $ has singular value decomposition
    $ V \Sigma ^T U ^T $. The matrices $ \Sigma $ and $ \Sigma ^T $ will have the same nonzero diagonal elements. Thus $ A $ and $ A ^T $ have the same nonzero singular values.
\end{proof} 

You will find it rather difficult to prove Generalized Cauchy-Schwartz Inequality. In fact it was introduced without any proof in YaXiang Yuan's masterpiece about Optimization \cite{WenYuSun2006}.

\begin{theorem}
    (Generalized Cauchy-Schwartz Inequality) Let A be an $ n \times n $  symmetric and positive deﬁnite matrix. If x and y are vectors in $ R ^ n $ , then the inequality

    \begin{equation*}
        \left| x ^T y \right| \le \left| \left| x \right| \right| _A \left| \left| y \right| \right| _{A ^{-1}}
    \end{equation*}
    holds, the equality holds if and only if x and $ A ^{-1} y $ are linearly dependent, where the vector norm is the ellipsoid norm, which is deﬁned as:

    \begin{equation*}
        \left| \left| v \right| \right| _A = \sqrt{ v ^T A v } .
    \end{equation*}
\end{theorem}

\begin{proof} 
    Because A is positive deﬁnite matrix, A and identity matrix I are said to be congruent, so there is a nondegenerate square matrix $ P $ such that: $ P ^T A P = I $.

    Let $ B = P ^{-1} $, we have $ A = \left( P ^T \right) ^{-1} P ^{-1} = \left( P ^{-1} \right) ^T P ^{-1} = B ^T B $ .
    
    Therefore,
    \begin{equation}
        \left| \left| x \right| \right| _A =  \sqrt {x ^T A x} = \sqrt {x ^T B ^T B x} = \sqrt { \left( B x \right) ^T \left( B x \right) } = \left| \left| B x \right| \right| _2
    \end{equation}

    Because A is positive deﬁnite, A is invertible, $ A ^{-1} = B ^{-1} B ^{-T} $. So
    \begin{align}
        \left| \left| y \right| \right| _{A ^{-1}}  &= \sqrt {y ^T A ^{-1} y}
        = \sqrt {y ^T B ^{-1} B ^{-T} y} \notag \\
        &= \sqrt {\left( B ^{-T} y \right) ^T \left( B ^{-T} y \right)} 
        = \left| \left| B ^{-T} y \right| \right|  _2
    \end{align}
    \begin{align}
        \left| x ^T y \right| &= \left| x ^T A A ^{-1} y\right| = \left| x ^T \left( B ^T B \right) \left( B ^{-1} B ^{-T} \right) y \right| \notag \\
        &= \left| \left( B x \right) ^T B B ^{-1} \left( B ^{-T} y \right) \right| = \left| \left( B x \right) ^T \left(B ^{-T} y \right) \right| \notag \\
        &\le \left| \left| B x \right| \right| _2 \left| \left| B ^{-T} y \right| \right| _2
    \end{align}

    It follows immediately from (2.8) and (2.9) that
    \begin{equation*}
        \left| x ^T y \right| \le \left| \left| x \right| \right| _A \left| \left| y \right| \right| _{A ^{-1}} 
    \end{equation*}

    Following Classical Cauchy-Schwartz Inequality, the equality of (2.10) holds if and only if $ B x $ and $ B ^{-T} y $ are linearly dependent. Let $ B x = r B ^{-T} y $, where r is a constant.
    \begin{equation*}
        x = r B ^{-1} B ^{-T} y = r \left( B ^T B \right) ^{-1} y = r A ^{-1} y
    \end{equation*}
\end{proof}

\begin{corollary}
    The Cauchy-Schwartz Inequality follows from Theorem 9 by setting $ A = I $, where $ I $ is an identity matrix.
\end{corollary}

\section{Find the steepest descent direction on the unit sphere under the ellipsoid norm}
\label{S3}
\setcounter{equation}{0}
\setcounter{table}{0}

We will find the solution of the minimization problem  for constrained optimization:
\begin{equation}
    min \quad g _k ^T d 
\end{equation}
\begin{equation*}
    s.t. \quad \left| \left| d \right| \right| _{B _k} = 1
\end{equation*}

where $ g _k = \nabla f \left( x _k \right) $, $ B _k $ is the approximation of the Hessian $ G _k $. They are known in the problem. $ d $ is a decision variable. We will constrain it on the unit sphere under the ellipsoid norm.

It follows from Generalized Cauchy-Schwartz Inequality that:
\begin{equation}
    \left| g _k ^T d \right| \le \left( d ^T B _k d \right) ^{\frac 1 2} \left( g _k ^T B _k ^{-1} g _k \right) ^{\frac 1 2}
\end{equation}
, the equality holds if and only if $ d $ and $ B _k ^{-1} g _k $  are linearly dependent. 

Therefore, the minimization of $g _k ^T d $ subject to constraints on its variables is:
\begin{equation*}
    - \left( d ^T B _k d \right) ^{\frac 1 2} \left( g _k ^T B _k ^{-1} g _k \right) ^{\frac 1 2} = -\left( g _k ^T B _k ^{-1} g _k \right) ^{\frac 1 2}
\end{equation*}
, which is a constant independent of decision variable $ d $.

So $ g _k ^T d $  is negative. And $ B _k $ is positive deﬁnite, so is $ B _k ^{-1} $.

\begin{equation*}
    d = r B _k ^{-1} g _k , \quad g _k ^T d = r \left( g _k ^T B _k ^{-1} g _k \right) 
\end{equation*}
. hence $ r $ has to be negative, while $ g _k ^T d $ is negative.

Hence direction $ d = - B _k ^{-1} g _k $. Because of $ \left| \left| d \right| \right| _{B _k} =1 $, $ d $ has to be a unit vector in the direction of $ - B _k ^{-1} g _k $:
\begin{align}
    d = \frac {- B _k ^{-1} g _k} {\left| \left| - B _k ^{-1} g _k \right| \right| _{B _k}} 
    &= \frac {- B _k ^{-1} g _k} {\left[ \left( B _k ^{-1} g _k \right) ^T B _k \left( B _k ^{-1} g _k \right) \right] ^{\frac 1 2}} \notag \\
    &=\frac {- B _k ^{-1} g _k} {\left[ g _k ^T \left( B _k ^{-1} \right) ^T B _k B _k ^{-1} g _k \right] ^{\frac 1 2}}
\end{align}

$ B _k $ is symmetric positive deﬁnite, so $ \left( B _k ^{-1} \right) ^T = \left( B _k ^T \right) ^{-1} = B _k ^{-1} $, therefore
\begin{equation}
    d = \frac {- B _k ^{-1} g _k} {\left( g _k ^T B _k ^{-1} g _k \right) ^{\frac 1 2}}
\end{equation}

It is the solution of the minimization problem for constrained optimization (3.1). So Quasi-Newton method is a steepest descent method under the ellipsoid norm.

\section{Conclusions}
\label{S4}
\setcounter{equation}{0}
\setcounter{table}{0}

    Quasi-Newton search directions provide an attractive alternative to Newton's method in that they do not require computation of the Hessian and yet still attain a superlinear rate of convergence. In Quasi-Newton method,  we require Hessian approximation to satisfy the secant equation.  Although Optimization software libraries in deep learning, contain a variety of quasi-Newton algorithms, most people are usually less familiar with the underlying mathematical theory.
    In this paper,  the Classical Cauchy-Schwartz Inequality is introduced, then more generalization are proposed. And it is seriously proved that  Quasi-Newton method is a steepest descent method under the ellipsoid norm.
    

\bibliographystyle{elsarticle-num}
\bibliography{}



\end{document}